\documentclass[12pt, reqno]{amsart}

\usepackage{amsmath, amsthm, amscd, amsfonts, amssymb, graphicx, color, bm}
\usepackage[bookmarksnumbered, colorlinks, plainpages]{hyperref}
\usepackage[scr=boondox]{mathalfa}
\usepackage{graphicx}
\usepackage{color}
\usepackage{CJK}
\usepackage{setspace}
\usepackage{verbatim}
\newtheorem{theorem}{Theorem}[section]
\newtheorem{lemma}[theorem]{Lemma}

\theoremstyle{definition}
\newtheorem{definition}[theorem]{Definition}

\theoremstyle{remark}

\numberwithin{equation}{section}
\makeatletter
\renewcommand{\subsection}{%
  \@startsection{subsection}{2}{\z@}%
    {-3.25ex\@plus -1ex \@minus -.2ex}%
    {1.5ex \@plus .2ex}{\normalfont\itshape}}
\makeatother    

\usepackage[
backend=biber,
style=numeric,
giveninits=true,
url=false,
doi=false,
isbn=false,
maxbibnames=99,
]{biblatex}
\begin{document}

\title{Affine dual Minkowski problem for general measures}
\author[H.Jin, C. Zhang]{
H\lowercase{ailin} J\lowercase{in}, C\lowercase{heng} Z\lowercase{hang}$^{\ast}$\\ \\
  D\lowercase{epartment of}
M\lowercase{athematics}\\ S\lowercase{uzhou} U\lowercase{niversity
of} S\lowercase{cience and} T\lowercase{echnology,}\\
S\lowercase{uzhou,} 215009 C\lowercase{hina}}

\begin{abstract}
    In 2025, the affine dual curvature measures $\widetilde{C}_m^a(K,\cdot)$ of convex body $K$ were introduced by Cai, Leng, Wu, and Xi, and the even Minkowski problem for the affine dual curvature measures was solved.  In this paper, the necessary and sufficient condition of affine dual Minkowski problem for general measures with $m>1$ are proposed.

\end{abstract}
\thanks{$^*$ Corresponding author.}
\thanks{\textbf{E-mail:} {\emph{jinhailin@usts.edu.cn; z\underline{  }c31415926@163.com}}}
\maketitle

\section{Introduction}
The \textit{classical Minkowski peoblem} asks: what conditons must a given measure on $S^{n-1}$ satisfy in order to there to exists a convex body whose surface area measure is the given measure? And, if the convex body exists, to what extent it is unique? These are the existence and uniqueness questions for the Minkowski problem. There are also regularity questions as well as stability questions for the Minkowski problem. This is an old problem with a long history filled with venerable names such as Minkowski\cite{zbMATH02673675,zbMATH02657630}, Aleksandrov\cite{zbMATH00930190}, Fenchel\&Jessen\cite{zbMATH02516933}, Nirenberg\cite{MR58265}, Caffarelli\cite{MR1038360}, Pogorelov\cite{zbMATH03486652}, and Yau\cite{zbMATH03565730}. See Schneider \cite{MR3155183} for a discussion. The Minkowski problem for the mixed area measures is called the \textit{Christoffel-Minkowski problem}. The Christoffel-Minkowski problem is still largely open. See Guan-Ma \cite{MR1961338} for a regular case. More result and progress relate Minkowski problem can refer \cite{MR2237290,MR4252759,MR4700389,Lutwak1993,MR4764744,MR3155183,MR2729006,MR3228445,MR3352764}.


During the last three decades, the study of new geometric measures greatly influenced the now vibrant Brunn-Minkowski theory. These new geometric measures include cone-voulume meausre\cite{MR3037788}, the $L_p$ surface area measure\cite{Lutwak1993}, and quite recently the dual curvature measure\cite{MR3573332}.


The dual Brunn-Minkowski theory was introduced by Lutwak in the mid-1970s. Many of fundamental geometric invariants (including mixed volumes and the quermassintegrals), and related formulas and geometric inequalities in the classical Brunn-Minkowski theory had counterparts in the dual theory. See e.g.\cite{Boeroeczky2019,MR3680945,MR2254308,MR3882970,MR3725875,MR3810248,Sheng2019,MR1254193,MR3605843}.  The Lutwak's work\cite{MR963487} in the 1980s remarkably demonstrated a relation between the dual Brunn-Minkowski theory and the long-open Busemann-Petty problem. The \textit{dual quermassintegral} $\widetilde{W}_{n-m}(K)$ of the body $K\subset\mathbb{R}^n$ is defined to be the mean value of the $m$-dimensional volumes of the central sections:
\begin{equation}\label{eq:1.1}
    \widetilde{W}_{n-m}(K)=\frac{\omega_n}{\omega_m}\int_{G(n,m)}\text{vol}_m(K\cap\xi)d\nu_m(\xi),\quad m=1,\cdots,n,
\end{equation}
where $G(n,m)$ denotes the Grassmannian of $m$ dimensional subspaces of $\mathbb{R}^n$, while $\nu_m$ denotes the Haar measure (the unique rotationally invariant probability measure) on $G(n,m)$, $\text{vol}_k$ denotes $k$ dimensional volume, and $\omega_m$ denotes the volume of the unit ball in $\mathbb{R}^m$. 
\subsection{\mdseries\itshape{Geometric measures and variation formulas}}\hfill

For convenience, let $\mathcal{K}^n$ denote the set of convex bodies (compact convex sets with nonempty interior), $\mathcal{K}_o^n$ and $\mathcal{K}_{(o)}^n$ denote the set of convex bodies containing the origin and the set of convex bodies containing the origin in its interiors respectively. 
It is natural to consider the differential of a functional on $\mathcal{K}^n$(or $\mathcal{K}_o^n$) to study its minimizers (or maximizer). Moreover, for a functional on $\mathcal{K}^n$, the study of its variation formulas is closely related to isoperimetric questions regarding the functional. Of central importance is the study of Minkowski-type problems by using variational methods. For example, the area measure $S_{n-1}(K,\cdot)$ of a convex body $K$ can be viewed as the differential of the volume $V(K)$ evaluated at the convex body $K$. Lutwak's $L_p$ surface area measures \cite{Lutwak1993} can be obtained by differentiating volume via $L_p$-combinations of bodies. 


\textbf{Dual curvature measures and their variation formulas.} In 2016, Huang-Lutwak-Yang-Zhang established the variation formula for dual quermassintegral

\begin{equation*}
    \frac{d}{dt}\Big|_{t=0}\widetilde{W}_{n-m}([K,f]_t)=\int_{S^{n-1}}f(v)d\widetilde{C}_m(K,v),
\end{equation*}
where for sufficiently small $\vert t\vert$, $f$ is the continuous function on $S^{n-1}$, the family of \textit{logarithmic Wulff shapes} $[K,f]_t$ defined by
\begin{equation*}
    [K,f]_t=\{x\in\mathbb{R}^n:\log(x\cdot v)_+\leq\log h_{K}(v)+tf(v),\enspace \text{for all}\enspace v\in S^{n-1}\}.
\end{equation*}
Here $s_+=\max\{s,0\}$ denotes the positive part of $s\in\mathbb{R}$, $h_{K}$ is the support function of the $K$. $\widetilde{C}_m(K,\cdot)$ is the \textit{dual curvature measure}, given by 
\begin{equation*}
    \widetilde{C}_m(K,\eta)=\frac{1}{n}\int_{\boldsymbol{\alpha}_{K}^*(\eta)}\rho_{K}(u)^mdu,
\end{equation*}
for Borel set $\eta\subset S^{n-1}$. Here $\rho_K(u)=\max\{\lambda>0:\lambda u\in K\}$, for all $u\in S^{n-1}$, is the \textit{radial function} of $K$, and $\boldsymbol{\alpha}_{K}^*(\eta)$ denotes the \textit{reverse radial Gauss image} of $K$, which is the set of all $u\in S^{n-1}$, such that the point $\rho_K(u)u\in\partial K$ has an outer unit normal to $\partial K$ that belongs to the set $\eta$.

Note that, when $m=n$ the dual curvature measure becomes the cone-volume measure, and when $m=0$ it truns out to be Aleksandrov's integral curvature of $K$'s \textit{polar body}
\begin{equation*}
    K^*=\{x\in\mathbb{R}^n:x\cdot y\leq1\enspace \text{for all}\enspace y\in K \};
\end{equation*}
see e.g.\cite{MR3851743}. The new approach in \cite{MR3851743} inspired work on fundamental problems in convex geometry from a larger perspective; see e.g., \cite{MR3882970,MR4040624,MR4156606,MR4252759}.

\subsection{\mdseries\itshape{Dual affine quermassintegral}}\hfill



 \textit{Dual affine quermassintegrals} were proposed by Lutwak, see \cite{MR3155183}. It was defined by letting
 $\widetilde{\Phi}_0(K)=V(K)$, and $\widetilde{\Phi}_n(K)=\omega_n$, while for $m=1,\cdots,n-1,$
 %
 \begin{equation*}
     \widetilde{\Phi}_{n-m}(K):=\frac{\omega_n}{\omega_m}\Bigg(\int_{G(n,m)}\text{vol}_m(K\cap\xi)^nd\nu_m(\xi)\Bigg)^{1/n}.
 \end{equation*}
 The normalization is to preserve the geometric meaning of these quantities.
 It was shown by Grinberg \cite{MR1125008} that $\widetilde{\Phi}_{n-m}$ is, as its names suggests, $\text{SL}(n)$ invariant. Isoperimetric inequalities for this affine invariants are stronger than their classical counterparts. When $m=n-1$, the affine isoperimetric inequalities for $\widetilde{\Phi}_1(k)$ are known as Busemann intersection inequality. For generic $m$, the affine isoperimetric inequality for $\widetilde{\Phi}_{n-m}(K)$\cite{MR1125008}, is
\begin{equation*}
    \widetilde{\Phi}_{n-m}\leq\omega_n^{n-m}V(K)^m,
\end{equation*}
where equality is achieved if and only if $K$ is an ellipsoid centered at the origin.

 The $n-$th moment of sections, $\widetilde{\Psi}_m$, defined for $K\in\mathcal{K}_o^n$, when $0<m<n$, by
 \begin{equation}\label{equation: 1.3}
       \widetilde{\Psi}_m(K)=\int_{G(n,m)}\text{vol}(K\cap\xi)^nd\nu_m(\xi).
 \end{equation}
 For completeness, set $\widetilde{\Psi}_0(K)=1$ and $\widetilde{\Psi}_n(K)=V(K)^n$. Obiviously, $\widetilde{\Psi}_m(K)=\widetilde{c}_{n,m}\widetilde{\Phi}_{n-m}(K)^n$, where
 \begin{equation}\label{eq:1.5}
     \widetilde{c}_{n,m}=\Big(\frac{\omega_m}{\omega_n}\Big)^n.
 \end{equation}

 Thus the $\widetilde{\Psi}_m$ is $\text{SL}(n)$ invariant.

 \subsection{\mdseries\itshape{Constructions of the affine geometric measures}}\hfill

\textbf{The affine dual curvature measure.} Suppose $m\in\{1,\cdots,n-1\}$ and $K\in\mathcal{K}_{(o)}^n$. The \textit{affine dual curvature measure} of $K$ is defined by
 \begin{equation*}
     \widetilde{C}_{m}^a(K,\eta)=\int_{\bm{\alpha}_K^*(\eta)}\rho_K(u)^m\mathcal{R}_m^*(\text{vol}_m(K\cap\cdot)^{n-1})(u)du,
 \end{equation*}
 for each Borel set $\eta\subset S^{n-1}$, where $\text{vol}_m(K\cap\cdot)$ means the continuous function $\xi\mapsto\text{vol}_m(K\cap\xi)^{n-1}$ in $C(G(n,m))$ (see (\ref{equation: 2.1}) for the precise definition of the dual Radon transform $\mathcal{R}_m^*$). The affine dual curvature measures of bodies that are $\text{SL}(n)-$images of each other,  are themselves $\text{SL}(n)-$images
 of each other. More precisely, for each $\varphi\in\text{SL}(n),$
 \begin{equation*}
     \widetilde{C}_m^a(\varphi K,\cdot)=\varphi^{-1}\widetilde{C}_m^a(K,\cdot).
 \end{equation*}

 Cai et al.\cite{MR4874854} established the variational formula of $\widetilde{\Phi}_{n-m}(K)$.
\begin{theorem}\cite{MR4874854}
    For a convex body $K\in\mathcal{K}_{(o)}^n$, and for each $f\in C(S^{n-1})$,
    \begin{equation*}
        \frac{d}{dt}\Big|_{t=0}\Big[\widetilde{\Phi}_{n-m}([K,f]_t)^n\Big]=n\big(\frac{\omega_n}{\omega_m}\big)^n\int_{S^{n-1}}f(v)d\widetilde{C}_m^a(K,v),
    \end{equation*}
    where for sufficiently small $[t]$, the logarithmic Wulff shapes $[K,f]_t$ is given by
    \begin{equation*}
        [K,f]_t=\{x\in\mathbb{R}^n:\log(x\cdot v)_+\leq\log h_K(v)+tf(v)\},\quad \text{for all } v\in S^{n-1}\}.
    \end{equation*}
\end{theorem}

\subsection{\mdseries\itshape{Minkowski problems for geometric measures}}\hfill

\textbf{Affine dual Minkowski problem.} Suppose $m\in\{1,\cdots,n-1\}$ is fixed. Find necessary and sufficient conditions that a finite Borel $\mu$ on $S^{n-1}$ must satisfy such that there exists a convex body $K\in\mathcal{K}_{(o)}^n$ satisfy
\begin{equation*}
    \mu=\widetilde{C}_m^a(K,\cdot).
\end{equation*}

When $m=1$, $\widetilde{C}_m^a(K,\cdot)$ become the \textit{cone-volume measure} $V_K$, also denoted by $V(K,\cdot)$, defined for Borel sets $\eta\subset S^{n-1}$ by
\begin{equation}
    V_K(\eta)=\frac{1}{n}\int_{x\in\nu_K^{-1}(\eta)}x\cdot\nu_K(x)d\mathcal{H}^{n-1}(x)=V(K\cap c(\eta)),
\end{equation}
which is the volume of the cone $K\cap c(\eta)$, where $c(\eta)$ is the cone of rays eminating from the origin such that $\partial K\cap c(\eta)=\nu^{-1}(\eta)$ for the Borel set $\eta\subset S^{n-1}$. In fact, when $m=1$, and  $K$ is original symmetric, $\widetilde{C}_m^a(K,\cdot)$ is proportion to $n$-th dual curvature measure $\widetilde{C}_n(K,\cdot)$, and $\widetilde{C}_n(K,\cdot)$ is the cone-volume measure \cite{MR3573332}.
 The Minkowski problem for cone-volume measure is remains open except the even case.

 A finite Borel measure $\mu$ on $S^{n-1}$ is said to satisfy the \textit{strict subspace concentration inequality} if
 \begin{equation}\label{eq:1.3}
     \frac{\mu(S^{n-1}\cap\xi)}{|\mu|}<\frac{\text{dim}\xi}{n},\quad \text{for each proper subspace}\hspace{0.5em}\xi.
 \end{equation}
 Cai et al.\cite{MR4874854} solved the affine dual Minkowski porblem for even measures. They obtained the following theorem.
\begin{theorem}\cite{MR4874854}
    Suppose $m\in\{1,\cdots,n-1\}$. If $\mu$ is a finite even Borel measure on $S^{n-1}$ with positive mass, and $\mu$ satisfies the stirct subspace concentration inequality(\ref{eq:1.3}), then there exists an origin-symmetric convex body $K$ such that $\widetilde{C}_m^a(K,\cdot)=\mu.$
\end{theorem} 
\begin{definition}
    A finite set $U$ of unit vectors in $\mathbb{R}^n$ is said to be in \textit{general position} if $U$ is not contained in any closed hemisphere of $S^{n-1}$ and any $n$ elements of $U$ are linearly independent.
\end{definition}

In this paper, we will prove the necessary and sufficient condition of the affine dual Minkowski problem for general measures when $m>1$. We first prove the case for discrete measures and then prove the case for general measures by using the approximating method.
\begin{theorem}
    If $\gamma_1,\cdots,\gamma_N\in\mathbb{R}^+ n\geq3$ and the unit vectors $u_1,\cdots,u_N$ are in general position and $m\in\{2,\cdots,n-1\}$, then there exists a polytope $P\in\mathcal{K}_{(o)}^n$such that
    \begin{equation*}
        \widetilde{C}_m^a(P,\cdot)=\sum_{k=1}^N\gamma_k\delta_{u_k}.
    \end{equation*}
\end{theorem}
\begin{theorem}
    For each nonzero finite Borel measure $\mu$ on $S^{n-1}$, $n\geq3$ and $m\in\{2,\cdots,n-1\}$, there exists a convex body $K\in\mathcal{K}_{(o)}^n$, such that 
    \begin{equation*}
        \widetilde{C}_m^a(K,\cdot)=\mu
    \end{equation*}
    if and only if $\mu$ is not concentrated on any closed hemisphere.
\end{theorem}
\section{Preliminaries}

In this section, for quick later reference, we collect some basic facts about convex bodies. The books of Helgason \cite{MR1790156}, Schneider \cite{MR3155183},  Koldobsky \cite{MR2132704}, Gardner \cite{MR2251886} are good general references.

 Let $\mathbb{R}^n$ denote $n-$dimensional Euclidean space with canonical inner product $x\cdot y$, for $x,y\in\mathbb{R}^n$. Write $\Vert x\Vert=\sqrt{x\cdot x}$ for the norm of $x$. Let $S^{n-1}$ denote the unit sphere in $\mathbb{R}^n$. We normalize $x$ as $\langle x\rangle=x/\Vert x\Vert$ for $x\in\mathbb{R}^n\backslash\{0\}$. If $x_1,\cdots,x_{n-1}\in\mathbb{R}^n$ are linearly independent, we write $x_1\wedge\cdots\wedge x_{n-1}$ for the outer product for $x_1,\cdots,x_{n-1}.$
 If $1\leq i \leq n$, $K$ is a Borel subset of $\mathbb{R}^n$, and $K$ is contained in an $i-$dimensional affine subspace of $\mathbb{R}^n$ but not in any affine subspace of lower dimension, then let $|K|$ be the $i-$dimensional Lebesgue measure of $K$. The $n$ dimensional Euclidean unit ball is denoted by $B_n$. The $m$ dimensional volume of the unit ball in $\mathbb{R}^m$ is denoted by $\omega_m$.  For a finite measure $\mu$ on $S^{n-1}$, we shall write $|\mu|$ for its total mass; that is $|\mu|=\mu(S^{n-1})$ . We denote the support of a measure $\mu$ by $\text{supp}(\mu).$ The distance between two points $x,y$ denoted by $d(x,y).$

\subsection{\mdseries\itshape{Convex bodies}}\hfill

As introduced in the previous section, the set of convex bodies in $\mathbb{R}^n$ is denoted by $\mathcal{K}^n$, the set of convex bodies containing the origin is denoted by $\mathcal{K}_o^n$, the set of convex bodies containing the origin in their interiors is denoted by $\mathcal{K}_{(o)}^n$.

The \textit{support function}, $h_K:S^{n-1}\to\mathbb{R}$, of a nonempty compact convex set $K\in\mathcal{K}^n$, is the continuous function on $S^{n-1}$, defined by 
\begin{equation*}
    h_{K}(v)=\max\{x\cdot v:x\in K\}.
\end{equation*}
$h_K(v)$ also denoted by $h(K,v)$. A nonempty, compact, convex set is uniquely determined by its support function. The collection of nonempty compact convex sets can be viewed as a metric space with the \textit{Hausdorff metric}, where the distance  $d(K,L)$ between $K,L$ is the max-norm $\Vert h_K-h_L\Vert_\infty$.

For a convex body $K$ in $\mathbb{R}^n$, and $u\in S^{n-1}$, the \textit{support hyperplane} $H(K,u)$ in direction $u$ is defined by 
 \begin{equation*}
     H(K,u)=\{x\in\mathbb{R}^n:x\cdot u =h(K,u)\}.
 \end{equation*}
the \textit{support set} $F(K,u)$ in direction $u$ is defined by 
\begin{equation*}
    F(K,u)=K\cap H(K,u).
\end{equation*} 
Let half spaces $H^-(u,a)$ be
\begin{equation*}
    H^-(u,a)=\{x\in \mathbb{R}^n:x\cdot u\leq a\}.
\end{equation*}
For a compact $K\subset\mathbb{R}^n$, the diameter of $K$ is defined by
\begin{equation*}
    \mathrm{diam}(K)=\max\{|x-y|:x,y\in K\}.
\end{equation*}

Let $\mathcal{P}$ be the set of polytopes in $\mathbb{R}^n$. If $u_1,\cdots,u_N\in S^{n-1}$ were not contained in any closed hemisphere, let $\mathcal{P}(u_1,\cdots,u_N)$ be the subset of $\mathcal{P}$ such that a polytope $P\in\mathcal{P}(u_1,\cdots,u_N)$ if
\begin{equation*}
    P=\bigcap_{k=1}^NH^-(P,u_k).
\end{equation*}
Obviously, if $P\in\mathcal{P}(u_1,\cdots,u_N)$, then $P$ has at most $N$ facets, and the outer unit normals of $P$ are a subset of $\{u_1,\cdots,u_N\}$. Let $\mathcal{P}_N(u_1,\cdots,u_N)$ be the subset of $\mathcal{P}(u_1,\cdots,u_N)$ such that a polytope $P\in\mathcal{P}_N(u_1,\cdots,u_N)$ if $P\in\mathcal{P}(u_1,\cdots,u_N)$, and $P$ has exactly $N$ facets.
\subsection{\mdseries\itshape{Radial Gauss image and its reverse}}\hfill

For $K\in\mathcal{K}^n$ and $z\in\partial K$, we write $N(K,z)$ to denote the \textit{normal cone} at $z$; namely,
\begin{equation*}
    N(K,z)=\{y\in\mathbb{R}^n:y\cdot(x-z)\leq0\text{ for }x\in K\}.
\end{equation*}
If $z\in\text{int}K$, then simply $N(K,z)=\{o\}$.

Let $K\in\mathcal{K}^n$ with $\text{int}K\neq\emptyset$. We recall that the so-called singular point $z\in\partial K$ of $\text{dim}N(K,z)\geq2$ form a Borel set of zero $\mathcal{H}^{n-1}$ measure, and hence its complement, the set of smooth points denoted by $\partial'K$ is also a Borel set. For $z\in\partial'K,$ we write $\nu_K(z)$ to denote the unique exterior normal at $z$. In addition, for any $z\in\partial K$, we define the \textit{Gauss image} of $K$ at $z$ as
\begin{equation*}
    \bm{\nu}_K(z)=N(K,z)\cap S^{n-1}.
\end{equation*}

For a Borel set $\eta\subset S^{n-1},$ we define the \textit{inverse Gauss image} of $K$ as
\begin{equation*}
    \bm{\nu}_K^{-1}(\eta)=\bigcup_{u\in\eta}F(K,u),
\end{equation*}
which is the set of all $z\in\partial K$ with $N(K,z)\cap\eta\neq\emptyset.$ 
Define the dual of $N(K,o)$ as
\begin{equation*}
    N(K,o)^*=\{y\in\mathbb{R}^n:y\cdot x\leq0\text{ for }x\in N(K,o)\}=\{\lambda x:\lambda\geq0\text{ and }x\in K\}.
\end{equation*}
For $K\in\mathcal{K}_o^n$ and $\eta\subset S^{n-1}$, the \textit{radial Gauss image} $\bm{\alpha}_K(\eta)$ is the set of outer unit normals of $K$ emanating from the boundary points $\rho_K(u)u$, for some $u\in\eta$; that is,
\begin{equation*}
    \bm{\alpha}_K(\eta)=\bigcup_{u\in\eta}\{v\in S^{n-1}:\rho_K(u)u\cdot v=h_K(v)\}.
\end{equation*}
The \textit{reverse radial Gauss image} $\bm{\alpha}_K^*(\eta)$ is the set of all radial directions $u\in S^{n-1}$, such that an element in $\eta$ is an outer unit normal to $\partial K$ at the point $\rho_K(u)u\in\partial K$; that is,
\begin{equation*}
    \bm{\alpha}_K^*(\eta)=\bigcup_{v\in\eta}\{u\in S^{n-1}:\rho_K(u)u\cdot v=h_K(v)\}.
\end{equation*}
Denote $\omega_K\subset S^{n-1}$ as the set of $u\in S^{n-1}$ such that $\bm{\alpha}_K(\{u\})$ contains more than one point; that is, the point $\rho_K(u)u\in\partial K$ has more than one outer unit normal. We now define the \textit{radial Gauss map}
\begin{equation*}
    \alpha_K:S^{n-1}\setminus\omega_K\longrightarrow S^{n-1},
\end{equation*}
satisfying $\bm{\alpha}_K(\{u\})=\{\alpha_K(u)\}$. It is well known that $\mathcal{H}^{n-1}(\omega_K)=0$(see \cite{MR3155183}).

\begin{lemma}\cite[Lemma 2.2]{Boeroeczky2019a}
If $K\in\mathcal{K}_o^n$ with $\mathrm{int} K\neq\emptyset$, then  
\begin{equation}\label{equation: 2.2}
    \bm{\alpha}_K^*(S^{n-1}\setminus N(K,o))=S^{n-1}\setminus(\mathrm{int}N(K,o)^*)
\end{equation}
and
\begin{equation}\label{equation: 2.3}
S^{n-1}\cap(\mathrm{int}N(K,o)^*)\subset\bm{\alpha}_K^*(S^{n-1}\setminus N(K,o))\subset S^{n-1}\cap N(K,o)^*.    
\end{equation}
\end{lemma}
\subsection{\mdseries\itshape{Radon transform and its dual}}\hfill

Let $m\in\{1,\cdots,n-1\}$. The \textit{m-dimensional Radon transform} $\mathcal{R}_m$ is a linear operator from $C(S^{n-1})$ into $C(G(n,m))$ given by
\begin{equation*}
    \mathcal{R}_mf(\xi)=\int_{S^{n-1}\cap\xi}f(w)dw.
\end{equation*}
Here $dw$ is used as an abbreviation for $d\mathcal{H}^{n-1}(w)$ on $S^{n-1}\cap\xi$. The \textit{m dimensional dual Radon transform $\mathcal{R}_m^*:C(G(n,m))\to C(S^{n-1})$} is defined by
\begin{equation}\label{equation: 2.1}
    \mathcal{R}_m^*F(u)=\frac{m\omega_m}{n\omega_n}\int_{G_{u}(n-1,m-1)}F(\text{span}\{u,\zeta\})d\nu_{m-1}(\zeta),
\end{equation}
where $G_u(n-1,m-1)$ denotes the Grassmannian of $m-1$ dimensional subspaces of $u^{\perp}$ and $\nu_{m-1}$ is the Haar measure on $G_u(n-1,m-1)$, and always, we use "Haar measure" to mean "Haar probability measure".
\subsection{\itshape The definition of affine dual curvature measure}

The \textit{affine dual curvature measure} of $K\in\mathcal{K}_{(o)}^n$ is defined by
\begin{equation*}
    \widetilde{C}_m^a(K,\eta)=\int_{\bm{\alpha}_K^*(\eta)}\rho_K^m(u)\mathcal{R}_m^*(|K\cap\cdot|^{n-1})(u)du,
\end{equation*}
for each Borel set $\eta\subset S^{n-1}.$ Note that $\widetilde{C}_{m}^a(K,\cdot)$ are $mn-$th homogeneous. We extended the affine dual curvature measure $\widetilde{C}_m^a(K,\cdot)$ to $K\in\mathcal{K}_o^n$ as follows, for a Borel set $\eta\subset S^{n-1}$,
\begin{equation}
\widetilde{C}_m^a(K,\eta)=0, \quad \text{if dim}K\leq n-1, \label{equation:7.1}
\end{equation}
and
\begin{equation}
\widetilde{C}_m^a(K,\eta)=\int_{\bm{\alpha}_K^*(\eta)}\rho_K^m(u)\mathcal{R}_m^*(|K\cap\cdot|^{n-1})(u)du, \quad\text{if int}K\neq\emptyset.\label{equation: 7.4}  
\end{equation}
Here, if $\text{int}K\neq\emptyset$, then $\rho_K(u)$ and $\mathcal{R}_m^*(|K\cap\cdot|^{n-1})(u)$ are continuous with respect to $u$ on $S^{n-1}\backslash\partial N(K,o)^*$. Moreover, by \cite[Lemma 2.3]{Boeroeczky2019a}, if $K\in\mathcal{K}_o^n$ with $\text{int} K\neq\emptyset$ and $\eta\subset S^{n-1}$ is a Borel set, then $\bm{\alpha}_K^*(\eta)\subset S^{n-1}$ is Lebesgue measurable. Therefore $\widetilde{C}_m^a(K,\cdot)$ is well-defined.

The following lemma was proved by Lin et al.\cite{lin2026lpminkowskiproblemsaffine}.
\begin{lemma}\label{lemma: 7.1}
    If $n\geq2$ and $m\in\{1,\cdots,n-1\}, K\in\mathcal{K}_o^n$ with $\mathrm{int}K\neq\emptyset$, and $g:S^{n-1}\to[0,\infty)$ is Borel measurable, then
    \begin{equation}\label{equation: 7.2}
        \int_{S^{n-1}}g(u)d\widetilde{C}_m^a(K,u)=\int_{S^{n-1}\cap(\mathrm{int}N(K,o)^*)}g(\alpha_K(u))\rho_K^m(u)\mathcal{R}_m^*(|K\cap\cdot|^{n-1})(u)d\mathcal{H}^{n-1}(u).
    \end{equation}
\end{lemma}
\subsection{The variation formula on $\widetilde{\Psi}_m$}

Let $h_0\in C^+(S^{n-1})$ and $f\in C(S^{n-1})$. For sufficiently small $\delta>0$ and $t\in(-\delta,\delta)$, let
\begin{equation*}\label{eqaution: 7.5}
    h_{f,t}(v)=h_0(v)e^{tf(v)+o(t,v)},
\end{equation*}
where $\lim_{t\to0}o(t,\cdot)/t=0$ uniformly on $S^{n-1}$. The \textit{family of Wulff shapes $K_{h_0,f,t}$} associated with $(h_0,f)$ is defined by 
\begin{equation}\label{equation: 2.4}
    K_{h_0,f,t}=\{x\in\mathbb{R}^n:x\cdot v\leq h_{f,t}\text{ for all }v\in S^{n-1}\},\quad t\in(-\delta,\delta),
\end{equation}
and abbreviate $K_{h_0,f,0}$ by $K_{h_0}$.
\begin{lemma}\cite[Theorem 3.1]{MR4874854}\label{lemma: 7.2}
    Let $m\in\{1,\cdots,n-1\}$, and let $K_{h_0,f,t}$ be the family of Wullf shapes defined by (\ref{equation: 2.4}). Then
    \begin{equation*}
        \lim_{t\to0}\frac{\widetilde{\Psi}_m(K_{h_0,f,t})-\widetilde{\Psi}_m(K_{h_0})}{t}=n\int_{S^{n-1}}f(u)d\widetilde{C}_m^a(K_{h_0},u).
    \end{equation*}
\end{lemma}
\subsection{The continuity of affine dual curvature measures}
In this section, we state the fact that the affine dual quermassintegral $\widetilde{\Psi}_m$ is a continuous function of $K\in\mathcal{K}_o^n$ with respect to the Hausdorff distance. Moreover, using the continuity of $\widetilde{\Psi}_m$, we state that the affine dual curvature measure $\widetilde{C}_m^a(K,\cdot)$ is weakly convergent on $K\in\mathcal{K}_o^n$.

The following three lemmas are due to Lin et al.\cite{lin2026lpminkowskiproblemsaffine}.
\begin{lemma}\cite[Lemma 3.7]{lin2026lpminkowskiproblemsaffine}\label{lemma: 7.3}
    If $K\in\mathcal{K}_o^n$, then
    \begin{equation}\label{equation: 7.5}
        \widetilde{\Psi}_m(K)=\frac{1}{m}\int_{S^{n-1}}\rho_K(u)^m\mathcal{R}_m^*(|K\cap\cdot|^{n-1})(u)du.
    \end{equation}
\end{lemma}

\begin{lemma}\cite[Lemma 3.8]{lin2026lpminkowskiproblemsaffine}\label{lemma: 7.4}
    For $m\in\{1,2,\cdots,n-1\}$, $\widetilde{\Psi}_m(K)$ is a continuous function of $K\in\mathcal{K}_o^n$ with respect to the Hausdorff distance.
\end{lemma}

\begin{lemma}\cite[Proposition 3.9]{lin2026lpminkowskiproblemsaffine}\label{lemma: 7.5}
    If $m\in\{1,\cdots,n-1\}$, and $\{K_j\}_{j\in\mathbb{N}},$ converges to $K$ for $K_j,K\in\mathcal{K}_o^n$, then $\widetilde{C}_m^a(K_j,\cdot)$ converges weakly to $\widetilde{C}_m^a(K,\cdot)$.
\end{lemma}
\section{An extreme problem}
Let $\gamma_1,\cdots,\gamma_N\in\mathbb{R}^+$, $u_1,\cdots,u_N$ are unit vectors in $\mathbb{R}^n$ and are in general position and $P\in\mathcal{P}(u_1,\cdots,u_N)$. Define $\Phi_p:\text{int}(P')\to\mathbb{R}$ by
\begin{equation*}
    \Phi_P(\xi)=\sum_{k=1}^N\frac{1}{2}\gamma_k\log(h(P,u_k)^2-(\xi\cdot u_k)^2),
\end{equation*}
where $P'=\{\xi\in P:h(P,u_k)^2-(\xi\cdot u_k)^2>0,k=1,\cdots,N\}$. Then the origin is contained in $P'$ and $P'$ has nonempty interior. Obviously, the unique point $\xi(P)\in\text{int}(P)$ such that 
\begin{equation*}
    \Phi_P(\xi(P))=\max_{\xi\in\text{Int}(P')}\Phi_P(\xi)
\end{equation*}
is $o$.

If $P_i\in\mathcal{P}(u_1,\cdots,u_N)$ and $P_i$ converges to a polytope $P$, then $P\in\mathcal{P}(u_1,\cdots,u_N)$. If the unit vectors $u_1,\cdots,u_N$ are in general position, $P_i\in\mathcal{P}(u_1,\cdots,u_N)$ and $P_i$ converges to a polytope $P$, then 
\begin{equation}\label{eq:3.3}
    \lim_{i\to\infty}\Phi_{P_i}(\xi(P_i))=\Phi_P(\xi(P)).
\end{equation}

\begin{lemma}\label{le:3.1}
    If $\gamma_1,\cdots,\gamma_N\in\mathbb{R}^+$, the unit vectors $u_1,\cdots,u_N$ are in general position and there exists a $P\in\mathcal{P}_N(u_1,\cdots,u_N)$ with $\widetilde{\Psi}_m(P)=\sum_{i=1}^N\gamma_i$ such that 
    $\Phi_P(o)=\inf\Bigl\{\Phi_Q(o):Q\in\mathcal{P}_N(u_1,\cdots,u_N)\enspace\text{and}\enspace\widetilde{\Psi}_m(Q)=\sum_{k=1}^N\gamma_k\Big\}$.
    Then,
    \begin{equation*}
        \widetilde{C}_m^a(P,\cdot)=m\sum_{k=1}^N\gamma_k\delta_{u_k}.
    \end{equation*}
\end{lemma}
\begin{proof}
    By $mn-$th homogeneity of $\widetilde{\Psi}_m(P)$ and $\widetilde{C}_m^a(P,\cdot)$, it is sufficient to establish the lemma under the assumption that $\sum_{k=1}^N\gamma_k=1$.

    For $\delta_1,\cdots,\delta_N\in\mathbb{R}$, choose $\vert t\vert$ small enough so that the polytope $P_t$ defined by 
    \begin{equation*}
        P_t=\mathop{\cap}\limits_{i=1}^N\{x:x\cdot u_i\leq h(P,u_i)e^{\frac{\delta_i}{h(P,u_i)}t}\}
    \end{equation*}
    has exactly $N$ facets. Then,
    \begin{equation*}
        \log h_{P_t}(u)=\log h_{P}(u)+\frac{\delta_i}{h_P}t,\quad \text{for }u\in\{u_1,\cdots,u_N\} 
    \end{equation*}
    and by Lemma \ref{lemma: 7.2},
    \begin{equation*}
        \lim_{t\to 0}\frac{\widetilde{\Psi}_m(P_t)-\widetilde{\Psi}_m(P)}{t}=\sum_{k=1}^N\frac{n\delta_k}{h_P(u_k)}\widetilde{C}_m^a(P,u_k).
     \end{equation*}
    Let $\lambda(t)=\widetilde{\Psi}_m(P_t)^{-\frac{1}{mn}}$, then $\lambda(t)P_t\in \mathcal{P}_N(u_1,\cdots,u_N)$, $\widetilde{\Psi}_m(\lambda(t)P_t)=1$. Then
    \begin{equation}\label{eq:3.1}
        \lambda'(0)=-\frac{1}{m}\sum_{k=1}^{N}\frac{\delta_k}{h_P(u_k)}\widetilde{C}_m^a(P,u_k).
    \end{equation}
    Let $\xi(t)=\xi(\lambda(t)P_t)=o$, and $\Phi(t)=\Phi_{\lambda(t)P_t}(\xi(\lambda(t)P_t))$, then we have
    \begin{align}\label{eq:3.2}
        \Phi(t)&=\max_{\xi\in\lambda(t)P_t}\frac{1}{2}\sum_{k=1}^N\gamma_k\log(h(\lambda(t)P_t,u_k)^2-(\xi\cdot u_k)^2)\notag\\  &=\sum_{k=1}^N\gamma_k\log(\lambda(t)h(P_t,u_k)).
    \end{align}
    From the fact that $\Phi(0)$ is a minimizer of $\Phi(t)$, (\ref{eq:3.1}), the fact $\sum_{k=1}^N\gamma_k=1$, we have
    \begin{align*}
        0&=\Phi'(0)\\
        &=\sum_{k=1}^N\frac{\gamma_k}{\lambda(0)h(P,u_k)}\big(\lambda'(0)h(P,u_k)+\lambda(0)\frac{d}{dt}h(P_t,u_k)\big|_{t=0}\big)\\
        &=\sum_{k=1}^N\gamma_k\lambda'(0)+\frac{\gamma_k\delta_k}{h(P,u_k)}\\
        &=\lambda'(0)+\sum_{k=1}^N\frac{\gamma_k\delta_k}{h(P,u_k)}\\
        &=-\frac{1}{m}\sum_{k=1}^N\frac{\delta_k}{h(P,u_k)}\widetilde{C}_m^a(P,u_k)+\sum_{k=1}^N\frac{\gamma_k}{h(P,u_k)}\delta_k\\
        &=\sum_{k=1}^N\Big(\frac{\gamma_k}{h(P,u_k)}-\frac{\widetilde{C}_m^a(P,u_k)}{mh(P,u_k)}\Big)\delta_k.
    \end{align*}
    Since $\delta_1,\cdots,\delta_N$ are arbitrary, $\gamma_k=\frac{1}{m}\widetilde{C}_m^a(P,u_k)$ for $k=1,\cdots,N$.
\end{proof}

\begin{lemma}\cite[Lemma 4.1]{MR3228445}\label{le:4.2}
  If the unit vectors $u_1,\cdots,u_N$ are in general position and $P\in\mathcal{P}(u_1,\cdots,u_N)$, then $F(P,u_i)$ is either a point or a facet for all $1\leq i\leq N$. Moreover, if $n\geq3$ and $F(P,u_i)$ is a facet, then the outer unit normals of $F(P,u_i)$(in $H(P,u_i)$) are in general position. 
\end{lemma}
\begin{lemma}\label{lemma:3.2}
    If $K\in\mathcal{K}_o^n, \mathrm{int}K\neq\emptyset$. Then for $x\in\mathbb{R}^n,$ with $x+cK\in\mathcal{K}_o^n$ and $c\geq1$, we have
    $\widetilde{\Psi}_m(x+cK)\to\infty$ as $c\to\infty.$
\end{lemma}
\begin{proof}
     Let $x_0\in\mathbb{R}^n$ such that 
     \begin{equation*}
         \widetilde{\Psi}_m(x_0+K)=\min\{\widetilde{\Psi}_m(x+K),x+K\in\mathcal{K}_o^n\}.
     \end{equation*}
     Then
     \begin{equation*}
           \widetilde{\Psi}_m(cx_0+cK)=\min\{\widetilde{\Psi}_m(x+cK):x+cK\in\mathcal{K}_o^n\}.
     \end{equation*}
   So
    \begin{equation*}
        \widetilde{\Psi}_m(x+cK)\geq\widetilde{\Psi}_m(cx_0+cK)=c^{mn}\widetilde{\Psi}(x_0+K)\to\infty,\quad c\to\infty.
    \end{equation*}
    
\end{proof}
\begin{lemma}\label{lemma:3.3}
For $\{P_i\}_{i\in\mathbb{N}}\subset\mathcal{K}_o^n,\mathrm{int}(P_i)\neq\emptyset$, let
\begin{equation*}
    d_i=\frac{\min\{\widetilde{\Psi}_m(x+P_i):x+P_i\in\mathcal{K}_o^n\}}{\max\{\widetilde{\Psi}_m(y+P_i):y+P_i\in\mathcal{K}_o^n\}}.
\end{equation*}
If
\begin{equation*}
    \inf_{i\in\mathbb{N}}\{\min_{x_i}\widetilde{\Psi}_m(x_i+P_i)\}>0,
\end{equation*}
then
\begin{equation*}
    d_0:=\inf\{ d_i: i=1,2, \cdots\} >0.
\end{equation*}
\end{lemma}
\begin{proof}
    For $P\in\mathcal{K}_o^n,\text{int}(P)\neq\emptyset$, Since 
    \begin{align*}
        &\frac{\min\{\widetilde{\Psi}_m(P+x):P+x\in\mathcal{K}_o^n\}}{\max\{\widetilde{\Psi}_m(P+y):P+y\in\mathcal{K}_o^n\}}\\
        =&\frac{\min\{\widetilde{\Psi}_m((P+x)/\text{diam}(P+x)):P+x\in\mathcal{K}_o^n\}}{\max\{\widetilde{\Psi}_m((P+y)/\text{diam}(P+y)):P+y\in\mathcal{K}_o^n\}},
    \end{align*}
    we can assume $\text{diam}(P_i)=1$. Suppose 
    \begin{equation*}
        \min\{ \widetilde{\Psi}_m(x+P_{i}):x+P_i\in\mathcal{K}_o^n\}=\widetilde{\Psi}_m(x_i+P_i)
    \end{equation*}
    and
    \begin{equation*}
         \max \{\widetilde{\Psi}_m(y+P_{i}):y+P_i\in\mathcal{K}_o^n\}=\widetilde{\Psi}_m(y_i+P_i).
    \end{equation*}
     Then
     \begin{equation*}
d_i
= \frac{\widetilde{\Psi}_m(x_i+P_i)}{\widetilde{\Psi}_m(y_i+P_i)} 
>\frac{\widetilde{\Psi}_m(x_i+P_i)}{\widetilde{\Psi}_m(2B^n)}> c > 0
\end{equation*}
     for some constant $c$, so $d_0>0$.
\end{proof}
\begin{theorem}\label{th:4.3}
    If the unit vectors $u_1,\cdots,u_N$ are in general position, $P_i\in\mathcal{P}(u_1,\cdots,u_N)$, $i=1,2,\cdots$ and $\widetilde{\Psi}_m(P_i)=1$, then 
    $\{P_i\}_{i=1}^{\infty}$ is bounded.
\end{theorem}
\begin{proof}
    We only need to prove that if the unit vectors $u_1,\cdots,u_N$ are in general position, for $P_i\in \mathcal{P}(u_1,\cdots,u_N)$,  $\{\mathrm{diam}(P_i)\}$ is unbounded, then $\{\widetilde{\Psi}_m(P_i)\}$ is not bounded.  Futhermore, we can assume $P_i\in\mathcal{P}_N(u_1,\cdots, u_N)$, otherwise let $P'_{it}\in\mathcal{P}_N(u_1,\cdots, u_N),\Psi_m(P'_{it})=1$ and $P'_{it}\to P_i, t\to\infty$, then by assumption, $\{P'_{it}\}_{i,t=1}^\infty$ is bounded. If $\{P_i\}_{i=1}^\infty$ is unbounded, then there exists a subsequence $\{P_{i_k}\}$ of $\{P_i\}$ satisfy $\text{diam}(P_{i_k})\to\infty, k\to\infty$, then for $\{P'_{i_ki_k}\}_{k=1}^\infty\subset\{P'_{it}\}_{i,t=1}^\infty$, we have $\text{diam}(P'_{i_ki_k})\to\infty,k\to\infty$, a contradiction.

    For each $j=1,\cdots,N$, if the volume of facets $\{\text{vol}_{n-1}(F(P_i,u_j))\}_{i\in\mathbb{N}}$ is unbounded and then $\{\widetilde{\Psi}_m(P_i)\}$ is obviously unbounded. If there is at least one $k\in\{1,\cdots,N\}$ such that the volume $\{\text{vol}_{n-1}(F(P_i,u_{k}))\}_{i\in\mathbb{N}}$ is bounded, let $K=\{u_k\}$ such that for $u_k\in K$, $\{\text{vol}_{n-1}F(P_i,u_k)\}_{i\in\mathbb{N}}$ is bounded, and $J=\{u_j\}$ such that for $u_j\in J$, $\{\text{vol}_{n-1}F(P_i,u_j)\}_{i\in\mathbb{N}}$ is unbounded. Let $l_{ij}=\text{vol}_{n-1}(F(P_i,u_j))$, we claim that $J$ are not concentrated on any closed hemisphere, if not, suppose $J$ are concentrated on closed hemisphere $H^-(v,0)\cap S^{n-1}$. $J$ can not concentrated on subspace $v^\perp$, otherwise, 
\begin{equation*}
\text{vol}_{n-1}(P_i\cap v^\perp)\to\infty. 
\end{equation*}
The support vectors of $P_i$ are $u_i,\cdots,u_N$, the number of which are finite.  There is relation
\begin{equation*}
\sum_{u_k\notin v^\perp}\text{vol}_{n-1}(F(P_i,u_k))>2\text{vol}_{n-1}(P_i\cap v^\perp)
\end{equation*}
but each $\{\text{vol}_{n-1}(F(P_i,u_k))\}_{i\in\mathbb{N}}$ is bounded for $u_k\notin v^\perp$, a contradiction. 
    
    According to classical theorem of Minkowski problem,we have
    \begin{equation*}
        \sum_{u_j\in H^-(v,0)}l_{ij}u_{ij}+\sum_{u_k\notin H^-(v,0)}l_{ik}u_{ik}=0,
    \end{equation*}
    that is
    \begin{equation*}
        \Big|\sum_{u_j\in H^-(v,0)}l_{ij}u_{ij}\Big|=\Big|\sum_{u_k\notin H^-(v,0)}l_{ik}u_{ik}\Big|,
    \end{equation*}
    the left side is unbounded as $i\to\infty$ while the right side is bounded as $i\to\infty$, a contradiction.
   Let 
   \begin{equation*}
       P_{1i}:=\bigcap_{j\in J}H^-(u_j,h(P_i,u_j)),
   \end{equation*}
   then as $\text{diam}(F(P_i,u_j)),u_j\in J$ is unbounded, $\widetilde{\Psi}_m(P_{1i})\to\infty,i\to\infty$. Let $P_{2i}=c_iP_{1i}+x_i$ such that $P_{2i}$ is inscribed in $P_i$.
   
   Let
   \begin{equation*}
       d_i=\frac{\min\{\widetilde{\Psi}_m\{x+P_{1i}):x+P_{1i}\in\mathcal{K}_o^n\}}{\max\{\widetilde{\Psi}_m(y+P_{1i}):y+P_{1i}\in\mathcal{K}_o^n\}},\quad \text{for} \enspace x_i+P_{1i},y_i+P_{1i}\in\mathcal{K}_o^n
   \end{equation*}
   and 
   \begin{equation*}
       d_0=\inf d_i.
   \end{equation*}
   
   We can assume $o\in P_{2i}$, otherwise we translate $P_i$ to $P_i+x_i$ so that $o\in P_{2i}+x_i$, and then prove $\widetilde{\Psi}_m(P_{2i}+x_i)\to\infty,$ hence $\widetilde{\Psi}_m(P_i+x_i)\to\infty,i\to\infty$, then by Lemma \ref{lemma:3.2}, $\widetilde{\Psi}_m(P_i)$ also tend to infinite to obtain contradiction.
   
   Since $\{\text{diam}(F(P_i,u_k))\}_{i\in\mathbb{N}},u_k\in K$ is bounded, $d(P_{1i},P_{2i})<c_0$ for some constant $c_0$, then there is a constant $c>0$ such that  $c_i\geq c$ and
   \begin{align*}
       \widetilde{\Psi}_m(P_{2i})=&\widetilde{\Psi}_m(x_i+c_iP_{1i})\\
       \geq&\min\{\widetilde{\Psi}_m(x+c_iP_{1i}):x+c_iP_{1i}\in\mathcal{K}_o^n\}\\
       \geq&d_0\max\{\widetilde{\Psi}_m(x+c_iP_{1i}):x+c_iP_{1i}\in\mathcal{K}_o^n\}\\
       \geq&d_0\widetilde{\Psi}_m(c_iP_{1i})\\
       =& d_0c_i^{mn}\widetilde{\Psi}_m(P_{1i})\\
       \geq&d_0c^{mn}\widetilde{\Psi}_m(P_{1i}).
   \end{align*}
    By Lemma \ref{lemma:3.2}, we have
    \begin{equation}\label{eq:3.4}
        \min\{\widetilde{\Psi}_m(x+P_{1i}):x+P_{1i}\in\mathcal{K}_o^n\}\to\infty,\quad i\to\infty.
    \end{equation}
    So by Lemma \ref{lemma:3.3}, $d_0>0$, and thus $\widetilde{\Psi}_m(P_{2i})\to\infty, i\to\infty.$ Since $P_{2i}\subset P_{i},$ we have $\widetilde{\Psi}_m(P_i)\to\infty,i\to\infty.$
\end{proof}
\section{The affine dual minkowski problem for general measures}
In this section we deal the affine dual Minkowski problem for general measures. We first deal the discrete case.
\begin{lemma}\label{le:5.1}
   If $\gamma_1,\cdots,\gamma_N$ are positive and the unit vectors $u_1,\cdots,u_N$ are in general position and $m\in\{2,\cdots,n-1\}$, then there exists a $P\in\mathcal{P}_N(u_1,\cdots,u_N)$ such that $\widetilde{\Psi}_m(P)=\sum_{k=1}^N\gamma_k$ and 
   \begin{equation*}
       \Phi_P(o)=\inf\{\Phi_Q(o):Q\in\mathcal{P}_N(u_1,\cdots,u_N)\enspace and\enspace\widetilde{\Psi}_m(Q)=\sum_{k=1}^N\gamma_k \},
   \end{equation*}
   where $\Phi_Q(o)=\sum_{k=1}^N\gamma_k\log(h(Q,u_k))$.
\end{lemma}
\begin{proof}
     It is sufficient to establish the lemma under the assumption that $\sum_{k=1}^N\gamma_k=1$.

    Thus, we can choose a sequence $P_i\in\mathcal{P}_N(u_1,\cdots,u_N)$ with $\widetilde{\Psi}_m(P_i)=1$ such that $\Phi_{P_i}(o)$ converges to 
    \begin{equation*}
        \inf\{\Phi_Q(o):Q\in\mathcal{P}_N(u_1,\cdots,u_N)\enspace and\enspace\widetilde{\Psi}_m(Q)=1 \}.
    \end{equation*}
    From Theorem \ref{th:4.3}, $\{P_i\}$ is bounded. Thus, from (\ref{eq:3.3}) and the Blaschke's selection theorem, there exists a subsequence of $\{P_i\}$ that converges to a polytope $P$ such that $P\in\mathcal{P}(u_1,\cdots,u_N)$, $\widetilde{\Psi}_m(P)=1$ and
    \begin{equation}\label{eq:5.1}
        \Phi_P(o)=\inf\{\Phi_Q(o):Q\in\mathcal{P}_N(u_1,\cdots,u_N)\enspace and\enspace\widetilde{\Psi}_m(Q)=1 \}.
    \end{equation}

     Next  we prove that $F(P,u_i)$ are facets for all $i=1,\cdots,N$. Otherwise, from Lemma \ref{le:4.2} and the fact that $N\geq n+1$, there exists $1\leq i_0<\cdots<i_k\leq N$ with $k\geq 0$ such that 
    \begin{equation*}
        F(P,u_i)
    \end{equation*}
    is a point for $i\in\{i_0,\cdots,i_k\}$ and is a facet of $P$ for $i\in\{1,\cdots,N\}\backslash\{i_0,\cdots,i_k\}.$

    Choose $\delta>0$ small enough so that the polytope
    \begin{equation*}
        P_{\delta}=P\cap\{x:x\cdot u_{i_0}\leq h(P,u_{i_0})-\delta\}
    \end{equation*}
    has exactly $(N-k)$ facets and 
    \begin{equation*}
        P\cap\{x:x\cdot u_{i_0}\geq h(P,u_{i_0})-\delta\}
    \end{equation*}
    is a truncated cone. 

    For $a>b$, we have the inequality 
    \begin{equation*}
        (a-b)^n>a^n-2na^{n-1}b.
    \end{equation*}
    Let $\max_{\xi\in G(n,m)}\text{vol}_m(P\cap\xi)=c_1$, then
    \begin{align*}
      \widetilde{\Psi}_m(P_{\delta})&=\int_{G(n,m)}\text{vol}_m(P_{\delta}\cap\xi)^nd\xi\\
      &=\int_{G(n,m)}\big(\text{vol}_m(P\cap\xi)-\text{vol}_m((P\backslash P_\delta)\cap\xi)\big)^nd\xi
      \\&>\int_{G(n,m)}\big(\text{vol}_m(P\cap\xi)-c_2\delta^m\big)^nd\xi\\
      &>\int_{G(n,m)}\text{vol}_m(P\cap\xi)^n-2nc_2\text{vol}_m(P\cap\xi)^{n-1}\delta^md\xi\\
      &>\widetilde{\Psi}_m(P)-2nc_1^{n-1}c_2\delta^m\\
      &=1-2nc_1^{n-1}c_2\delta^m,
    \end{align*}
    where $c_2$ is a constant that depends on $P$ and direction $u_{i_0}$. The first inequality is hold because there exists a constant $c>0$ such that $\sup_{\xi\in G(n,m)}\text{diam}((P\backslash P_\delta)\cap\xi)<c\delta$, so $\text{vol}_m((P\backslash P_\delta)\cap\xi)\leq \omega_m\text{diam}(((P\backslash P_\delta)\cap\xi)/2)^m<\omega_m(c/2)^m\delta^m=:c_2\delta^m$.

    Let $\delta$ be small enough so that $h(P,u_k)>\delta$ for all $k\in\{1,\cdots,N\}$, let $d_0=\text{diam}(P)$, and let 
    \begin{equation*}
        \lambda=\widetilde{\Psi}_m(P_\delta)^{-\frac{1}{mn}}<\frac{1}{(1-2nc_1^{n-1}c_2\delta^m)^{\frac{1}{mn}}}. 
    \end{equation*} 
    From this, the fact that $\sum_{k=1}^N\gamma_k=1$ and $d_0>h(P,u_{i_0})>0$, we have
    \begin{align*}
        &\prod_{k=1}^N(h(\lambda P_{\delta},u_k))^{\gamma_k}\\
        =&\lambda\prod_{k=1}^Nh(P_\delta,u_k)^{\gamma_k}\\
        =&\lambda\prod_{k\neq i_0}h(P,u_k)^{\gamma_k}\cdot\big(h(P,u_{i_0})-\delta)\big)^{\gamma_{i_0}}\\
        =&\lambda\prod_{k=1}^N h(P,u_k)^{\gamma_k}\Big(1-\frac{\delta}{h(P,u_{i_0})}\Big)^{\gamma_{i_0}}\\
        <&\prod_{k=1}^Nh(P,u_k)^{\gamma_k}\frac{(1-\frac{\delta}{d_0})^{\gamma_{i_0}}}{(1-2nc_1^{n-1}c_2\delta^m)^{\frac{1}{mn}}}.
    \end{align*}
    Let $g(\delta)=(1-2nc_1^{n-1}c_2\delta^m)^{\frac{1}{mn\gamma_{i_0}}}-1+\frac{\delta}{d_0}$, then $g(0)=0$ and if $m>1$,
    \begin{equation*}
        g'(\delta)=\frac{1}{d_0}+\frac{1}{mn\gamma_{i_0}}(1-2nc_1^{n-1}c_2\delta^m)^{\frac{1}{mn\gamma_{i_0}}-1}(-2nc_1^{n-1}c_2)m\delta^{m-1}>0
    \end{equation*}
    for small positive $\delta$. Thus, there exists a $\delta_0>0$ such that $g(\delta_0)>g(0)=0$ and 
    \begin{equation*}
        \frac{(1-\frac{\delta_0}{d_0})^{\gamma_{i_0}}}{(1-2nc_1^{n-1}c_2\delta_0^m)^{\frac{1}{mn}}}<1.
    \end{equation*}
    In this case,  
    $P_{\delta_0}$ has exactly $(N-k)$ facets, and 
    \begin{equation*}
        \prod_{k=1}^N(h(\lambda_0P_{\delta_0},u_k))^{\gamma_k}<\prod_{k=1}^N(h(P,u_k))^{\gamma_k},
    \end{equation*}
    where $\lambda_0=\widetilde{\Psi}_m(P_{\delta_0})^{-\frac{1}{mn}}$. Thus,
    \begin{equation*}
        \Phi_{\lambda_0P_{\delta_0}}(o)<\Phi_P(o).
    \end{equation*}
    Let $P_0=\lambda_0P_{\delta_0}$, then $P_0\in\mathcal{P}(u_1,\cdots,u_N)$, $\widetilde{\Psi}_m(P_0)=1$ and
    \begin{equation}\label{eq:5.2}
        \Phi_{P_0}(o)<\Phi_P(o).
    \end{equation}
    Since $P_0$ has $N-k$ facets, we can assume these normal of facets are $u_{k+1},\cdots,u_{N}$, and $F(P_0,u_1),\cdots,F(P_0,u_k)$ are points.
     Choose positive $\delta_j$ so that $\delta_j\to 0$ as $j\to\infty$,
    \begin{equation*}
        P_{\delta_j}=P_0\cap\Big(\cap_{i=1}^k\{x:x\cdot u_{i}\leq h(P_0,u_{i})-\delta_j\}\Big),
    \end{equation*}
    and $\lambda_jP_{\delta_j}\in\mathcal{P}_N(u_1,\cdots,u_N)$, where $\lambda_j=\widetilde{\Psi}_m(P_{\delta_j})^{-\frac{1}{mn}}$. Obviously, $\lambda_jP_{\delta_j}$ converges to $P_0$. From (\ref{eq:3.3}) and (\ref{eq:5.2}), we have
    \begin{align*}
        \lim_{j\to\infty}\Phi_{\lambda_jP_{\delta_j}}(o)&=\Phi_{P_0}(o)\\
        &<\Phi_P(o)\\
        &= \inf\{\Phi_Q(o):Q\in\mathcal{P}_N(u_1,\cdots,u_N) \enspace and\enspace\widetilde{\Psi}_m(Q)=1 \}.
    \end{align*}
    This is contradiction with (\ref{eq:5.1}). Therefore, $P\in\mathcal{P}_N(u_1,\cdots,u_N)$.
\end{proof}

\begin{theorem}\label{theorem: 5.2}
    If $\gamma_1,\cdots,\gamma_N\in\mathbb{R}^+$ and the unit vectors $u_1,\cdots,u_N$ are in general position and $m\in\{2,\cdots,n-1\}$, then there exists a polytope (containing the origin in its interior) such that
    \begin{equation*}
        \widetilde{C}_m^a(P,\cdot)=\sum_{k=1}^N\gamma_k\delta_{u_k}.
    \end{equation*}
\end{theorem}
    \begin{proof}
        From Lemma \ref{le:5.1}, there exists a $P_0\in\mathcal{P}_N(u_1,\cdots,u_N)$ with $\widetilde{\Psi}_m(P_0)=\sum_{k=1}^N\gamma_k$ such that
        \begin{equation*}
           \Phi_{P_0}(o)=\inf\{\Phi_Q(o):Q\in\mathcal{P}_N(u_1,\cdots,u_N)\enspace and\enspace\widetilde{\Psi}_m(Q)=\sum_{k=1}^N\gamma_k \}. 
        \end{equation*}
        From this and Lemma \ref{le:3.1}, we have
        \begin{equation*}
            \widetilde{C}_m^a(P_0,\cdot)=m\sum_{k=1}^N\gamma_k\delta_{u_k}.
        \end{equation*}
        Let $P=m^{-\frac{1}{mn}}P_0$, then
        \begin{equation*}
            \widetilde{C}_m^a(P)=\sum_{k=1}^N\gamma_k\delta_{u_k}.
        \end{equation*}
    \end{proof}

We now discuss the affine dual Minkowski problem for general measures.
\begin{lemma}\label{lemma:4.1}
    For any finite Borel measure $\mu$ there is a sequence of discrete measures $\{\mu_i\}_{i\in\mathbb{N}}$ converges weakly to it,  whose support are in general position.
\end{lemma}
\begin{proof}
    For such $\mu$ and $\epsilon>0$, there is a discrete measure $\bar{\mu}=\sum_{i=1}^N\gamma_i\delta_{u_i}$ with $\text{supp}(\bar{\mu})=\{u_1,\cdots u_N\}$ that are not contained in any closed hemisphere, and for any function $f\in C(S^{n-1})$, we have
    \begin{equation*}
        |\int_{S^{n-1}}f(u)d\mu(u)-\int_{S^{n-1}}f(u)d\bar{\mu}(u)|<\epsilon/2.
    \end{equation*}
  In particular, take $f=1$, we have $||\mu|-|\bar{\mu}||<\epsilon/2$. So for sufficiently small $\epsilon$, we have
  \begin{equation}\label{equation: 5.3}
  |\bar{\mu}|<2|\mu|.
  \end{equation}
  
  We can assume $\{u_{1},u_{2},\cdots,u_{n}\}$ are linearly independent. Next we construct a sequence of measures $\{\mu_i\}_{i\in\mathbb{N}}$ such that $\{\mu_i\}_{i\in\mathbb{N}}$ converges to $\bar{\mu}$ weakly and hence converges to $\mu$ weakly. For each $i$, we take $u_{il}=u_l$ for $l=1,\cdots, n.$
    For $l=n+1$, we take $u_{il}$  such that for any $1\leq l_1<\cdots<l_{n-1}\leq l-1$, $\{u_{il_1},u_{il_2},\cdots,u_{il_{n-1}},u_{il}\}$ is linearly independent and $|u_{il}-u_{l}|<1/i$. Similarly, we can take $u_{i,n+1},\cdots u_{i,N}$ such that $\mu_i=\sum_{j=1}^N\gamma_j\delta_{u_{ij}}$ are in general position.  For any continuous function $f(u)$ on $S^{n-1}$,  $j=1,\cdots,N$ and for every sufficiently large $i$ we have  $|f(u_j)-f(u_{ij})|<\epsilon/4|\mu|$ and by (\ref{equation: 5.3}),
    \begin{align*}
        &|\int_{S^{n-1}}f(u)d\bar{\mu}(u)-\int_{S^{n-1}}f(u)d\mu_i(u)|\\
        \leq&\sum_{j=1}^{N}|f(u_j)-f(u_{ij})|\gamma_j\\
        <&\sum_{j=1}^N\frac{\epsilon\gamma_j}{4|\mu|}=\frac{\epsilon|\bar{\mu}|}{4|\mu|}<\epsilon/2.
    \end{align*}
    Thus for any $f\in C(S^{n-1})$, 
    \begin{align*}
    &|\int_{S^{n-1}}f(u)d\mu_i(u)-\int_{S^{n-1}}f(u)d\mu(u)|\\
    \leq&|\int_{S^{n-1}}f(u)d\mu_i(u)-\int_{S^{n-1}}f(u)d\bar{\mu}(u)|+|\int_{S^{n-1}}f(u)d\bar{\mu}(u)-\int_{S^{n-1}}f(u)d\mu(u)|<\epsilon.
    \end{align*}
    Then $\{\mu_i\}$ are in general position and converges weakly to $\mu$.
\end{proof}
\begin{theorem}
    For each nonzero finite Borel measure $\mu$ on $S^{n-1}$ and $m\in\{2,\cdots,n-1\}$, there exists a convex body $K\in\mathcal{K}_{(o)}^n$, $\mathrm{int}K\neq\emptyset$, such that 
    \begin{equation*}
        \widetilde{C}_m^a(K,\cdot)=\mu
    \end{equation*}
    if and only if $\mu$ is not concentrated on any closed hemisphere.
\end{theorem}
\begin{proof}
\textit{sufficient part}

    For this $\mu$, by Lemma \ref{lemma:4.1}, there is a sequence of discrete measures $\{\mu_i\}$ whose support is  in general position converges weakly to it. Since $\mu_i$ is a discrete measure whose support are in general position, by Theorem \ref{theorem: 5.2}, there exists a polytope $P_i$ such that 
    \begin{equation*}
        \widetilde{C}_m^a(P_i,\cdot)=\mu_i
    \end{equation*}
    and $\mu_i$ converges to $\mu$ weakly.
    
    Since $\mu_i$ converges to $\mu$ weakly, $\widetilde{C}_m^a(P_i,S^{n-1})=|\mu_i|$ is bounded. Since the support of $\mu_i$ is in general positon, if $\{P_i\}_{i\in\mathbb{N}}$ is not bounded, by Theorem \ref{th:4.3}, $\widetilde{\Psi}_m(P_i)$ is unbounded, then by Lemma \ref{lemma: 7.3}, that leads $\widetilde{C}_m^a(P_i,S^{n-1})$ is unbounded, so $\{P_i\}_{i\in\mathbb{N}}$ is bounded. By Blaschke's selection theorem, there is a subsequence of $\{P_i\}_{i\in\mathbb{N}}$,which we still denote by $\{P_i\}_{i\in\mathbb{N}}$, converge to a convex set $K$ contain the origin. If $K$ has empty interior, then $\widetilde{C}_{m}^a(P_i,S^{n-1})$ will tend to 0, which contradict to $\widetilde{C}_{m}^a(P_i,S^{n-1})=|\mu_i|\to|\mu|\neq0$. So $K\in\mathcal{K}_o^n$. By Lemma \ref{lemma: 7.5}, $\mu_i=\widetilde{C}_m^a(P_i,\cdot)\to\widetilde{C}_m^a(K,\cdot),i\to\infty,$ but $\mu_i\to\mu$, so we obtain $\widetilde{C}_m^a(K,\cdot)=\mu$. 
    
    Next, if $o\in \partial K$, for $\eta\nsubseteq N(K,o)^*$, the integrand of $\widetilde{C}_m^a(K,\eta)$ is 0. Since $N(K,o)^*$ contained in some closed hemisphere, the support of $\widetilde{C}_m^a(K,\cdot)$, so the support of $\mu$ will concentrated on some closed hemisphere, a contradiction. So $K\in\mathcal{K}_{(o)}^n$.
     
\textit{necessary part}
     
     Suppose there exists $K\in \mathcal{K}_{(o)}^n$ such that $\widetilde{C}_m^a(K,\cdot)=\mu$, we need to prove the support of $\mu$ are not concentrated on any closed hemisphere. Conversely suppose $\mu$ concentrated on some closed hemisphere
     \begin{equation*}
     H=\{u\in S^{n-1}: u\cdot w\geq 0\},\quad w\in S^{n-1},
     \end{equation*} 
     that is $\mu(S^{n-1}\backslash H)=0$, thus $\widetilde{C}_m^a(K,S^{n-1}\backslash H)=0$. Since $K\in\mathcal{K}_{(o)}^n$, there exists $r>0$ such that $rB_n\subset K$. Thus for any $u\in S^{n-1}$,
     \begin{equation*}
      \rho_K(u)\geq r>0,\quad \rho_K(u)^m\geq r^m>0,
     \end{equation*}
     and
    \begin{equation*}
    \mathcal{R}_m^*(|K\cap\cdot|^{n-1})(u)=\frac{m\omega_m}{n\omega_n}\int_{G_u(n-1,m-1)}\text{vol}_m(K\cap\text{span}\{u,\zeta\})^{n-1}d\nu_{m-1}(\zeta)\geq\frac{m\omega_m^n}{n\omega_n}r^{m(n-1)}>0.
    \end{equation*}
    So the integrand of $\widetilde{C}_m^a(K,\cdot)$ has positive lower bound on $S^{n-1}$. By definition (\ref{equation: 7.4}), 
    \begin{equation*}
    \widetilde{C}_m^a(K, S^{n-1}\backslash H)=0
    \end{equation*}
    imply
    \begin{equation*}
    \mathcal{H}^{n-1}(\bm{\alpha}_K^*(S^{n-1}\backslash H))=0.
    \end{equation*}
    That is almost all $u\in S^{n-1}$ has
    \begin{equation}\label{equation: 6.1}
    \bm{\alpha}_K(u)\in H.
    \end{equation} 
   Let $\Phi(u)=\rho_K(u)u$, on the regular points we have 
   \begin{equation*}
   \nu_K(\Phi(u))=\alpha_K(u).
   \end{equation*}
    Let 
    \begin{equation*}
   E=\bm{\alpha}_K^*(S^{n-1}\backslash H),\quad \mathcal{H}^{n-1}(E)=0.
   \end{equation*}
   On the regular points part,
   \begin{equation*}
   \{z\in\partial'K:\nu_K(z)\in S^{n-1}\backslash H\}=\Phi(E\cap\{u:\bm{\alpha}_K(u)\text{ is single point set}\}).
   \end{equation*}
   Since $\rho_K$ is Lipschitz, so $\Phi(u)=\rho_K(u)u$ is also Lipschitz on $S^{n-1}$. Lipschitz mapping maps zero measure set to zero measure set. So
   \begin{equation*}
   \mathcal{H}^{n-1}(\Phi(E))=0.
   \end{equation*}
   The measure of singuler points set $\partial K\backslash \partial K'$ is 0, so
   \begin{equation*}
   S_K(S^{n-1}\backslash H)=\mathcal{H}^{n-1}(\{z\in\partial'K:\nu_K(z)\in S^{n-1}\backslash H\})=0,
   \end{equation*} 
   a contradiction.
\end{proof}

\noindent\textbf{Declaration of competing interest}\\
The authors declare that they have no conflict of interest.\\
\noindent\textbf{Funding} \\
The authors declare that no funds, grants, or other support were received during the preparation of this manuscript.


@Article{MR3228445,
  author     = {Zhu, Guangxian},
  journal    = {Adv. Math.},
  title      = {The logarithmic {M}inkowski problem for polytopes},
  year       = {2014},
  issn       = {0001-8708,1090-2082},
  pages      = {909--931},
  volume     = {262},
  doi        = {10.1016/j.aim.2014.06.004},
  fjournal   = {Advances in Mathematics},
  mrclass    = {52A40},
  mrnumber   = {3228445},
  mrreviewer = {Thomas\ Wannerer},
  url        = {https://doi.org/10.1016/j.aim.2014.06.004},
}

@Article{MR4874854,
  author     = {Cai, Xiaxing and Leng, Gangsong and Wu, Yuchi and Xi, Dongmeng},
  journal    = {Adv. Math.},
  title      = {Affine dual {M}inkowski problems},
  year       = {2025},
  issn       = {0001-8708,1090-2082},
  pages      = {Paper No. 110184, 35},
  volume     = {467},
  doi        = {10.1016/j.aim.2025.110184},
  fjournal   = {Advances in Mathematics},
  mrclass    = {52A20 (52A30 52A40)},
  mrnumber   = {4874854},
  mrreviewer = {Jos\'e\ Vidal-N\'u\~nez},
  url        = {https://doi.org/10.1016/j.aim.2025.110184},
}

@Article{MR58265,
  author     = {Nirenberg, Louis},
  journal    = {Comm. Pure Appl. Math.},
  title      = {The {W}eyl and {M}inkowski problems in differential geometry in the large},
  year       = {1953},
  issn       = {0010-3640,1097-0312},
  pages      = {337--394},
  volume     = {6},
  doi        = {10.1002/cpa.3160060303},
  fjournal   = {Communications on Pure and Applied Mathematics},
  mrclass    = {53.0X},
  mrnumber   = {58265},
  mrreviewer = {H.\ Busemann},
  url        = {https://doi.org/10.1002/cpa.3160060303},
}

@Article{MR1038360,
  author     = {Caffarelli, Luis A.},
  journal    = {Ann. of Math. (2)},
  title      = {Interior {$W^{2,p}$} estimates for solutions of the {M}onge-{A}mp\`ere equation},
  year       = {1990},
  issn       = {0003-486X,1939-8980},
  number     = {1},
  pages      = {135--150},
  volume     = {131},
  doi        = {10.2307/1971510},
  fjournal   = {Annals of Mathematics. Second Series},
  mrclass    = {35B65 (35B45 35J60)},
  mrnumber   = {1038360},
  mrreviewer = {John\ Urbas},
  url        = {https://doi.org/10.2307/1971510},
}

@Article{MR3037788,
  author     = {B\"or\"oczky, K\'aroly J. and Lutwak, Erwin and Yang, Deane and Zhang, Gaoyong},
  journal    = {J. Amer. Math. Soc.},
  title      = {The logarithmic {M}inkowski problem},
  year       = {2013},
  issn       = {0894-0347,1088-6834},
  number     = {3},
  pages      = {831--852},
  volume     = {26},
  doi        = {10.1090/S0894-0347-2012-00741-3},
  fjournal   = {Journal of the American Mathematical Society},
  mrclass    = {52A40 (52A38)},
  mrnumber   = {3037788},
  mrreviewer = {Alina\ Stancu},
  url        = {https://doi.org/10.1090/S0894-0347-2012-00741-3},
}

@Article{Lutwak1993,
  author    = {Erwin Lutwak},
  journal   = {Journal of Differential Geometry},
  title     = {{The Brunn-Minkowski-Firey theory. I. Mixed volumes and the Minkowski problem}},
  year      = {1993},
  number    = {1},
  pages     = {131 -- 150},
  volume    = {38},
  doi       = {10.4310/jdg/1214454097},
  publisher = {Lehigh University},
  url       = {https://doi.org/10.4310/jdg/1214454097},
}

@Article{MR3573332,
  author     = {Huang, Yong and Lutwak, Erwin and Yang, Deane and Zhang, Gaoyong},
  journal    = {Acta Math.},
  title      = {Geometric measures in the dual {B}runn-{M}inkowski theory and their associated {M}inkowski problems},
  year       = {2016},
  issn       = {0001-5962,1871-2509},
  number     = {2},
  pages      = {325--388},
  volume     = {216},
  doi        = {10.1007/s11511-016-0140-6},
  fjournal   = {Acta Mathematica},
  mrclass    = {52A38 (35J20 35J96)},
  mrnumber   = {3573332},
  mrreviewer = {Ai-jun\ Li},
  url        = {https://doi.org/10.1007/s11511-016-0140-6},
}

@Article{MR963487,
  author     = {Lutwak, Erwin},
  journal    = {Adv. in Math.},
  title      = {Intersection bodies and dual mixed volumes},
  year       = {1988},
  issn       = {0001-8708},
  number     = {2},
  pages      = {232--261},
  volume     = {71},
  doi        = {10.1016/0001-8708(88)90077-1},
  fjournal   = {Advances in Mathematics},
  mrclass    = {52A40 (52A20)},
  mrnumber   = {963487},
  mrreviewer = {Jane\ R.\ Sangwine-Yager},
  url        = {https://doi.org/10.1016/0001-8708(88)90077-1},
}

@Article{MR3851743,
  author     = {Huang, Yong and Lutwak, Erwin and Yang, Deane and Zhang, Gaoyong},
  journal    = {J. Differential Geom.},
  title      = {The {$L_p$}-{A}leksandrov problem for {$L_p$}-integral curvature},
  year       = {2018},
  issn       = {0022-040X,1945-743X},
  number     = {1},
  pages      = {1--29},
  volume     = {110},
  doi        = {10.4310/jdg/1536285625},
  fjournal   = {Journal of Differential Geometry},
  mrclass    = {52A38 (35J20 35J96)},
  mrnumber   = {3851743},
  mrreviewer = {Ai-jun\ Li},
  url        = {https://doi.org/10.4310/jdg/1536285625},
}

@Article{MR3882970,
  author     = {Gardner, Richard J. and Hug, Daniel and Weil, Wolfgang and Xing, Sudan and Ye, Deping},
  journal    = {Calc. Var. Partial Differential Equations},
  title      = {General volumes in the {O}rlicz-{B}runn-{M}inkowski theory and a related {M}inkowski problem {I}},
  year       = {2019},
  issn       = {0944-2669,1432-0835},
  number     = {1},
  pages      = {Paper No. 12, 35},
  volume     = {58},
  doi        = {10.1007/s00526-018-1449-0},
  fjournal   = {Calculus of Variations and Partial Differential Equations},
  mrclass    = {52A20 (52A30 52A39 52A40)},
  mrnumber   = {3882970},
  mrreviewer = {Yiming\ Zhao},
  url        = {https://doi.org/10.1007/s00526-018-1449-0},
}

@Article{MR4040624,
  author     = {Gardner, Richard J. and Hug, Daniel and Xing, Sudan and Ye, Deping},
  journal    = {Calc. Var. Partial Differential Equations},
  title      = {General volumes in the {O}rlicz-{B}runn-{M}inkowski theory and a related {M}inkowski problem {II}},
  year       = {2020},
  issn       = {0944-2669,1432-0835},
  number     = {1},
  pages      = {Paper No. 15, 33},
  volume     = {59},
  doi        = {10.1007/s00526-019-1657-2},
  fjournal   = {Calculus of Variations and Partial Differential Equations},
  mrclass    = {52A20 (52A30 52A39 52A40)},
  mrnumber   = {4040624},
  mrreviewer = {Yiming\ Zhao},
  url        = {https://doi.org/10.1007/s00526-019-1657-2},
}

@Article{MR4156606,
  author     = {B\"or\"oczky, K\'aroly J. and Lutwak, Erwin and Yang, Deane and Zhang, Gaoyong and Zhao, Yiming},
  journal    = {Comm. Pure Appl. Math.},
  title      = {The {G}auss image problem},
  year       = {2020},
  issn       = {0010-3640,1097-0312},
  number     = {7},
  pages      = {1406--1452},
  volume     = {73},
  doi        = {10.1002/cpa.21898},
  fjournal   = {Communications on Pure and Applied Mathematics},
  mrclass    = {52A20},
  mrnumber   = {4156606},
  mrreviewer = {Julio\ Bernu\'es},
  url        = {https://doi.org/10.1002/cpa.21898},
}

@Article{MR4252759,
  author     = {Huang, Yong and Xi, Dongmeng and Zhao, Yiming},
  journal    = {Adv. Math.},
  title      = {The {M}inkowski problem in {G}aussian probability space},
  year       = {2021},
  issn       = {0001-8708,1090-2082},
  pages      = {Paper No. 107769, 36},
  volume     = {385},
  doi        = {10.1016/j.aim.2021.107769},
  fjournal   = {Advances in Mathematics},
  mrclass    = {52A40 (35J96 52A38)},
  mrnumber   = {4252759},
  mrreviewer = {Andrea\ Colesanti},
  url        = {https://doi.org/10.1016/j.aim.2021.107769},
}

@Book{MR3155183,
  author     = {Schneider, Rolf},
  publisher  = {Cambridge University Press, Cambridge},
  title      = {Convex bodies: the {B}runn-{M}inkowski theory},
  year       = {2014},
  edition    = {expanded},
  isbn       = {978-1-107-60101-7},
  series     = {Encyclopedia of Mathematics and its Applications},
  volume     = {151},
  mrclass    = {52-02 (52A20 52A39)},
  mrnumber   = {3155183},
  mrreviewer = {Andrea\ Colesanti},
  pages      = {xxii+736},
}

@Article{MR1125008,
  author     = {Grinberg, Eric L.},
  journal    = {Math. Ann.},
  title      = {Isoperimetric inequalities and identities for {$k$}-dimensional cross-sections of convex bodies},
  year       = {1991},
  issn       = {0025-5831,1432-1807},
  number     = {1},
  pages      = {75--86},
  volume     = {291},
  doi        = {10.1007/BF01445191},
  fjournal   = {Mathematische Annalen},
  mrclass    = {52A40 (52A20 52A38)},
  mrnumber   = {1125008},
  mrreviewer = {L.\ A.\ Santal\'o},
  url        = {https://doi.org/10.1007/BF01445191},
}

@Article{MR1961338,
  author     = {Guan, Pengfei and Ma, Xi-Nan},
  journal    = {Invent. Math.},
  title      = {The {C}hristoffel-{M}inkowski problem. {I}. {C}onvexity of solutions of a {H}essian equation},
  year       = {2003},
  issn       = {0020-9910,1432-1297},
  number     = {3},
  pages      = {553--577},
  volume     = {151},
  doi        = {10.1007/s00222-002-0259-2},
  fjournal   = {Inventiones Mathematicae},
  mrclass    = {35J60 (53C65)},
  mrnumber   = {1961338},
  mrreviewer = {Fabiana\ Leoni},
  url        = {https://doi.org/10.1007/s00222-002-0259-2},
}

@Book{MR1790156,
  author    = {Helgason, Sigurdur},
  publisher = {American Mathematical Society, Providence, RI},
  title     = {Groups and geometric analysis},
  year      = {2000},
  isbn      = {0-8218-2673-5},
  note      = {Integral geometry, invariant differential operators, and spherical functions, Corrected reprint of the 1984 original},
  series    = {Mathematical Surveys and Monographs},
  volume    = {83},
  doi       = {10.1090/surv/083},
  mrclass   = {22-02 (22E30 22E46 43A85 43A90 44A12 53-02 58-02)},
  mrnumber  = {1790156},
  pages     = {xxii+667},
  url       = {https://doi.org/10.1090/surv/083},
}

@Book{MR2132704,
  author     = {Koldobsky, Alexander},
  publisher  = {American Mathematical Society, Providence, RI},
  title      = {Fourier analysis in convex geometry},
  year       = {2005},
  isbn       = {0-8218-3787-7},
  series     = {Mathematical Surveys and Monographs},
  volume     = {116},
  doi        = {10.1090/surv/116},
  mrclass    = {42A38 (46B20 52A20 52A38)},
  mrnumber   = {2132704},
  mrreviewer = {Keith\ Ball},
  pages      = {vi+170},
  url        = {https://doi.org/10.1090/surv/116},
}

@Book{MR2251886,
  author    = {Gardner, Richard J.},
  publisher = {Cambridge University Press, New York},
  title     = {Geometric tomography},
  year      = {2006},
  edition   = {Second},
  isbn      = {0-521; 0-521-68493-5},
  series    = {Encyclopedia of Mathematics and its Applications},
  volume    = {58},
  doi       = {10.1017/CBO9781107341029},
  mrclass   = {52A22 (44A12 92C55)},
  mrnumber  = {2251886},
  pages     = {xxii+492},
  url       = {https://doi.org/10.1017/CBO9781107341029},
}

@Article{Boeroeczky2019a,
  author     = {B\"or\"oczky, K\'aroly J. and Fodor, Ferenc},
  journal    = {J. Differential Equations},
  title      = {The {$L_p$} dual {M}inkowski problem for {$p>1$} and {$q>0$}},
  year       = {2019},
  issn       = {0022-0396,1090-2732},
  number     = {12},
  pages      = {7980--8033},
  volume     = {266},
  doi        = {10.1016/j.jde.2018.12.020},
  fjournal   = {Journal of Differential Equations},
  mrclass    = {52A40 (35J60 35J96)},
  mrnumber   = {3944247},
  mrreviewer = {Andrea\ Colesanti},
  url        = {https://doi.org/10.1016/j.jde.2018.12.020},
}

@Article{MR2237290,
  author     = {Guan, Pengfei and Ma, Xi-Nan and Zhou, Feng},
  journal    = {Comm. Pure Appl. Math.},
  title      = {The {C}hristofel-{M}inkowski problem. {III}. {E}xistence and convexity of admissible solutions},
  year       = {2006},
  issn       = {0010-3640,1097-0312},
  number     = {9},
  pages      = {1352--1376},
  volume     = {59},
  doi        = {10.1002/cpa.20118},
  fjournal   = {Communications on Pure and Applied Mathematics},
  mrclass    = {35J60 (35B05 53C65)},
  mrnumber   = {2237290},
  mrreviewer = {Fabiana\ Leoni},
  url        = {https://doi.org/10.1002/cpa.20118},
}

@Article{MR4700389,
  author     = {Li, Ni and Ye, Deping and Zhu, Baocheng},
  journal    = {Math. Ann.},
  title      = {The dual {M}inkowski problem for unbounded closed convex sets},
  year       = {2024},
  issn       = {0025-5831,1432-1807},
  number     = {2},
  pages      = {2001--2039},
  volume     = {388},
  doi        = {10.1007/s00208-023-02570-5},
  fjournal   = {Mathematische Annalen},
  mrclass    = {52A20 (52A30 52A39 52A40)},
  mrnumber   = {4700389},
  mrreviewer = {Andrea\ Colesanti},
  url        = {https://doi.org/10.1007/s00208-023-02570-5},
}

@Article{MR4764744,
  author     = {Lutwak, Erwin and Xi, Dongmeng and Yang, Deane and Zhang, Gaoyong},
  journal    = {Comm. Pure Appl. Math.},
  title      = {Chord measures in integral geometry and their {M}inkowski problems},
  year       = {2024},
  issn       = {0010-3640,1097-0312},
  number     = {7},
  pages      = {3277--3330},
  volume     = {77},
  doi        = {10.1002/cpa.22190},
  fjournal   = {Communications on Pure and Applied Mathematics},
  mrclass    = {53C65 (52A20)},
  mrnumber   = {4764744},
  mrreviewer = {V.\ K.\ Ohanyan},
  url        = {https://doi.org/10.1002/cpa.22190},
}

@Article{MR2729006,
  author     = {Xiong, Ge},
  journal    = {Adv. Math.},
  title      = {Extremum problems for the cone volume functional of convex polytopes},
  year       = {2010},
  issn       = {0001-8708,1090-2082},
  number     = {6},
  pages      = {3214--3228},
  volume     = {225},
  doi        = {10.1016/j.aim.2010.05.016},
  fjournal   = {Advances in Mathematics},
  mrclass    = {52A40 (52B60)},
  mrnumber   = {2729006},
  mrreviewer = {Paolo\ Salani},
  url        = {https://doi.org/10.1016/j.aim.2010.05.016},
}

@Article{MR3352764,
  author     = {Zhu, Guangxian},
  journal    = {J. Funct. Anal.},
  title      = {The {$L_p$} {M}inkowski problem for polytopes for {$0<p<1$}},
  year       = {2015},
  issn       = {0022-1236,1096-0783},
  number     = {4},
  pages      = {1070--1094},
  volume     = {269},
  doi        = {10.1016/j.jfa.2015.05.007},
  fjournal   = {Journal of Functional Analysis},
  mrclass    = {52A40},
  mrnumber   = {3352764},
  mrreviewer = {Stefano\ Campi},
  url        = {https://doi.org/10.1016/j.jfa.2015.05.007},
}

@Article{Boeroeczky2019,
  author     = {B\"or\"oczky, K\'aroly J. and Lutwak, Erwin and Yang, Deane and Zhang, Gaoyong and Zhao, Yiming},
  journal    = {Adv. Math.},
  title      = {The dual {M}inkowski problem for symmetric convex bodies},
  year       = {2019},
  issn       = {0001-8708,1090-2082},
  pages      = {106805, 30},
  volume     = {356},
  doi        = {10.1016/j.aim.2019.106805},
  fjournal   = {Advances in Mathematics},
  mrclass    = {52A20},
  mrnumber   = {4008522},
  mrreviewer = {Mar\'ia\ A.\ Hern\'andez Cifre},
  url        = {https://doi.org/10.1016/j.aim.2019.106805},
}

@Article{MR3680945,
  author   = {Chen, Shibing and Li, Qi-rui and Zhu, Guangxian},
  journal  = {J. Differential Equations},
  title    = {On the {$L_p$} {M}onge-{A}mp\`ere equation},
  year     = {2017},
  issn     = {0022-0396,1090-2732},
  number   = {8},
  pages    = {4997--5011},
  volume   = {263},
  doi      = {10.1016/j.jde.2017.06.007},
  fjournal = {Journal of Differential Equations},
  mrclass  = {35J96 (35J20 35J60 52A40)},
  mrnumber = {3680945},
  url      = {https://doi.org/10.1016/j.jde.2017.06.007},
}

@Article{MR2254308,
  author     = {Chou, Kai-Seng and Wang, Xu-Jia},
  journal    = {Adv. Math.},
  title      = {The {$L_p$}-{M}inkowski problem and the {M}inkowski problem in centroaffine geometry},
  year       = {2006},
  issn       = {0001-8708,1090-2082},
  number     = {1},
  pages      = {33--83},
  volume     = {205},
  doi        = {10.1016/j.aim.2005.07.004},
  fjournal   = {Advances in Mathematics},
  mrclass    = {52A38 (35J20 35J60 52A21 52A39 52A40 53A15)},
  mrnumber   = {2254308},
  mrreviewer = {Wolfgang\ Lusky},
  url        = {https://doi.org/10.1016/j.aim.2005.07.004},
}

@Article{MR3725875,
  author     = {Henk, Martin and Pollehn, Hannes},
  journal    = {Adv. Math.},
  title      = {Necessary subspace concentration conditions for the even dual {M}inkowski problem},
  year       = {2018},
  issn       = {0001-8708,1090-2082},
  pages      = {114--141},
  volume     = {323},
  doi        = {10.1016/j.aim.2017.10.037},
  fjournal   = {Advances in Mathematics},
  mrclass    = {52A40 (52A38)},
  mrnumber   = {3725875},
  mrreviewer = {Mar\'ia\ A.\ Hern\'andez Cifre},
  url        = {https://doi.org/10.1016/j.aim.2017.10.037},
}

@Article{MR3810248,
  author     = {Huang, Yong and Zhao, Yiming},
  journal    = {Adv. Math.},
  title      = {On the {$L_p$} dual {M}inkowski problem},
  year       = {2018},
  issn       = {0001-8708,1090-2082},
  pages      = {57--84},
  volume     = {332},
  doi        = {10.1016/j.aim.2018.05.002},
  fjournal   = {Advances in Mathematics},
  mrclass    = {52A40 (35J20 35J96 52A38)},
  mrnumber   = {3810248},
  mrreviewer = {Mar\'ia\ A.\ Hern\'andez Cifre},
  url        = {https://doi.org/10.1016/j.aim.2018.05.002},
}

@Article{Sheng2019,
  author  = {Sheng, Weimin and Wang, Xu-Jia},
  journal = {Journal of the European Mathematical Society},
  title   = {Flow by Gauss curvature to the Aleksandrov and dual Minkowski problems},
  year    = {2019},
  month   = {12},
  pages   = {893-923},
  volume  = {22},
  doi     = {10.4171/JEMS/936},
}

@Article{MR1254193,
  author     = {Zhang, Gao Yong},
  journal    = {Trans. Amer. Math. Soc.},
  title      = {Centered bodies and dual mixed volumes},
  year       = {1994},
  issn       = {0002-9947,1088-6850},
  number     = {2},
  pages      = {777--801},
  volume     = {345},
  doi        = {10.2307/2154998},
  fjournal   = {Transactions of the American Mathematical Society},
  mrclass    = {52A39},
  mrnumber   = {1254193},
  mrreviewer = {W.\ J.\ Firey},
  url        = {https://doi.org/10.2307/2154998},
}

@Article{MR3605843,
  author     = {Zhao, Yiming},
  journal    = {Calc. Var. Partial Differential Equations},
  title      = {The dual {M}inkowski problem for negative indices},
  year       = {2017},
  issn       = {0944-2669,1432-0835},
  number     = {2},
  pages      = {Paper No. 18, 16},
  volume     = {56},
  doi        = {10.1007/s00526-017-1124-x},
  fjournal   = {Calculus of Variations and Partial Differential Equations},
  mrclass    = {52A40 (49Q20)},
  mrnumber   = {3605843},
  mrreviewer = {Ai-jun\ Li},
  url        = {https://doi.org/10.1007/s00526-017-1124-x},
}

@Misc{lin2026lpminkowskiproblemsaffine,
  author        = {Youjiang Lin and Yuchi Wu},
  title         = {The $L_p$ Minkowski problems on affine dual quermassintegrals},
  year          = {2026},
  archiveprefix = {arXiv},
  eprint        = {2504.12117},
  primaryclass  = {math.MG},
  url           = {https://arxiv.org/abs/2504.12117},
}

@Article{zbMATH02673675,
  author   = {Minkowski, H.},
  journal  = {Nachr. Ges. Wiss. G{\"o}ttingen, Math.-Phys. Kl.},
  title    = {Allgemeine {Lehrs{\"a}tze} {\"u}ber die convexen {Polyeder}.},
  year     = {1897},
  pages    = {198--219},
  volume   = {1897},
  fjournal = {Nachrichten von der Gesellschaft der Wissenschaften zu G{\"o}ttingen. Mathematisch-Physikalische Klasse},
  jfm      = {28.0427.01},
  language = {German},
  url      = {https://eudml.org/doc/58391},
  zbmath   = {2673675},
}

@Article{zbMATH02657630,
  author   = {Minkowski, H.},
  journal  = {Math. Ann.},
  title    = {Volumen und {Oberfl{\"a}che}.},
  year     = {1903},
  issn     = {0025-5831},
  pages    = {447--495},
  volume   = {57},
  doi      = {10.1007/BF01445180},
  fjournal = {Mathematische Annalen},
  jfm      = {34.0649.01},
  language = {German},
  url      = {https://eudml.org/doc/158108},
  zbmath   = {2657630},
}

@Book{zbMATH00930190,
  author    = {Alexandrov, A. D.},
  editor    = {Reshetnyak, Yu. G. and Kutateladze, S. S.},
  publisher = {Amsterdam: Gordon {and} Breach Publishers},
  title     = {Selected works. {Part} 1: {Selected} scientific papers. {Ed}. by {Yu}. {G}. {Reshetnyak} and {S}. {S}. {Kutateladze}, transl. from the {Russian} by {P}. {S}. {V}. {Naidu}},
  year      = {1996},
  isbn      = {2-88124-984-1},
  series    = {Class. Sov. Math.},
  volume    = {4},
  fseries   = {Classics of Soviet Mathematics},
  issn      = {0743-9199},
  language  = {English},
  zbl       = {0960.01035},
  zbmath    = {930190},
}

@Misc{zbMATH02516933,
  author       = {Fenchel, W. and Jessen, B.},
  howpublished = {Danske {Videnks}. {Selsk}. {Math}.-fys. {Medd}. 16, {Nr}. 3, 31 {S}. (1938).},
  title        = {Mengenfunktionen und konvexe {K{\"o}rper}.},
  year         = {1938},
  jfm          = {64.0733.05},
  language     = {German},
  zbmath       = {2516933},
}

@Book{zbMATH03486652,
  author    = {Pogorelov, A. V.},
  publisher = {American Mathematical Society (AMS), Providence, RI},
  title     = {Extrinsic geometry of convex surfaces. {Translated} from the {Russian} by {Israel} {Program} for {Scientific} {Translations}},
  year      = {1973},
  series    = {Transl. Math. Monogr.},
  volume    = {35},
  fseries   = {Translations of Mathematical Monographs},
  issn      = {0065-9282},
  language  = {English},
  zbl       = {0311.53067},
  zbmath    = {3486652},
}

@Article{zbMATH03565730,
  author   = {Cheng, Shiu-Yuen and Yau, Shing-Tung},
  journal  = {Commun. Pure Appl. Math.},
  title    = {On the regularity of the solution of the {{\(n\)}}-dimensional {Minkowski} problem},
  year     = {1976},
  issn     = {0010-3640},
  pages    = {495--516},
  volume   = {29},
  doi      = {10.1002/cpa.3160290504},
  fjournal = {Communications on Pure and Applied Mathematics},
  language = {English},
  zbl      = {0363.53030},
  zbmath   = {3565730},
}
\end{document}